\documentclass[12pt]{amsart}
\usepackage{amscd}
\usepackage[all]{xy}
\usepackage[notcite,notref]{showkeys}

\newcommand{\colim}{{\operatorname{colim}}}
\newcommand{\Gr}{{\operatorname{Gr}}}

\newcommand{\Hom}{\operatorname{Hom}}
\newcommand{\End}{\operatorname{End}}
\newcommand{\CH}{\operatorname{CH}}

\newcommand{\Pic}{\operatorname{Pic}}

\newcommand{\Spec}{\operatorname{Spec}}

\newcommand{\C}{\mathbf{C}}
\renewcommand{\L}{\mathbf{L}}
\renewcommand{\P}{\mathbf{P}}

\newcommand{\Q}{\mathbf{Q}}
\newcommand{\Z}{\mathbf{Z}}
\newcommand{\un}{\mathbf{1}}

 \newcommand{\sK}{\mathcal{K}}
 \newcommand{\sC}{\mathcal{C}}

\newcommand{\sM}{\mathcal{M}}
\newcommand{\sMrat}{\sM}
\newcommand{\sO}{\mathcal{O}}

\newcommand{\sJ}{\mathcal{J}}
\newcommand{\sI}{\mathcal{I}}
\newcommand{\sX}{\mathcal{X}}

 \newcommand{\sZ}{\mathcal{Z}}
 
\newcommand{\sV}{\mathcal{V}}

 \numberwithin{equation}{section}

\theoremstyle{plain}
\newtheorem{thm}[equation]{Theorem}
\newtheorem{prop}[equation]{Proposition}
\newtheorem{lm}[equation]{Lemma}
\newtheorem{cor}[equation]{Corollary}
\newtheorem{conj}[equation]{Conjecture}
\newtheorem{ques}[equation]{Question}

\theoremstyle{definition}
\newtheorem{defn}[equation]{Definition}
\newtheorem{ex}[equation]{Example}
\newtheorem{rk}[equation]{Remark}

\begin{document}

\title[\ ] 
{  Bloch's conjecture and valences of correspondences for    K3 surfaces}

 \author   { CLAUDIO PEDRINI}  \bigskip

  \date{\today}
\maketitle

 \begin{abstract}  Bloch's conjecture for a surface $X$ over an   algebraically closed field $k$ states that every  homologically trivial correspondence  $\Gamma $ acts as 0 on the Albanese kernel $T(X_{\Omega})$, where $\Omega $ is a universal domain containing $k$. Here we prove that, for a complex K3 surface $X$, Bloch's conjecture is equivalent to the existence of a valence for every correspondence. We also  give applications of this result to the case of a correspondence associated to an  automorphisms of finite order  and to the existence of constant cycle curves on $X$. Finally we show that Franchetta's conjecture, as stated by K.O'Grady, holds true for the family of polarized K3 surfacees of genus $g$, if $ 3 \le g \le 6$ 
\end{abstract} 
 
\section {introduction} 
 
 The existence of a suitable filtration for the Chow ring of every smooth projective variety over a field $k$, as conjectured by Bloch and Beilinson (see [MNP,Ch.7]), has many important consequences both in arithmetic and in geometry.  Apart from the trivial case of curves and some other particular cases,  this conjecture is still wide open.  Jannsen [Jan, 2.1] has shown that the existence of a Bloch -Beilinson filtration  for every smooth projective variety is equivalent  to Murre's conjectures on the existence of a suitable  Chow-K\"unneth decomposition of the   motives in $\sM_{rat}(k)$.\par
\noindent   Here $\sM_{rat}(k)$ denotes the (covariant ) category of Chow motives  with rational coefficients over the field $k$, which is  is a  $\Q$-linear, pseudoabelian, tensor category. If $X$ is a smooth irreducible , projective variety of dimension $d$  the motive of $X$ is $(X,\Delta_X)$. If $X$ and $Y$ are smooth projective varieties, then $\Hom_{\sM_{rat}}(h(X),h(Y))=A^d(X \times Y)$, where
  $A^*(X \times Y ) =CH^*(X\times Y)\otimes \Q$.  If $f : X \to Y$ is a morphism, then the correspondence $\Gamma_f \in A^2(X \times Y)$ is an element of $\Hom_{\sM_{rat}}(h(X),h(Y))$. \par 
\noindent   A  smooth irreducible projective surface (over any field  $k$) has  a refined Chow-K\"unneth decomposition  $\sum_{0 \le i \le 4}h_i(X)$
where $h_0(X) \simeq \un$, $h_4(X) \simeq \L^2$,  $h_2(X) =h^{alg}_2(X) +t_2(X)$ and $t_2(X)= (X,\pi^{tr}_2)$, see [KMP,2.2]. Here  
$$ \pi^{alg}_2(X)= \sum_{1 \le h \le \rho} \frac {[D_h \times D_h]}{D^2_h}. $$
\noindent with $\{D_h\}$ an orthogonal basis of $NS(X) \otimes \Q$ and
$\rho =\rho(X)$   the rank of $NS(X)$. Therefore 
$(X,\pi^{alg}_2)\simeq\mathbf{L}^{\oplus \rho}$. We also have
\begin{align*}
A_i(t_2(X))=& \pi^{tr}_2 A_i(X)=0    \  for  \    i \ne 0   \  ; \\
A_0(t_2(X))=&\Hom_{\sMrat}(\un,t_2(X)) \simeq T(X) 
\end{align*}  
\noindent where $T(X)$ is the Albanese kernel.The transcendental motive $t_2(X)$ is independent of the construction of the
Chow-K\"unneth decomposition, it is functorial for the action of
correspondences and a birational invariant for $X$.The motive $h(X)$ is finite dimensional, in the sense of Kimura iff $t_2(X)$ is evenly finite dimensional,i.e. $\wedge^n t_2(X) =0$, for some $n >0$, see \cite[Thm. 7.6.12]{KMP}. 
\noindent    A consequence of the Bloch-Beilinson's conjectures, or equivalently of Murre's conjectures , is the following Bloch's conjecture for surfaces.\par
\begin{conj} Let $X$ be a smooth projective surface over  a field $k$ of characteristic 0 .  Let $\Gamma \in A^2(X \times X)_{hom}$ be a  homologically trivial correspondence   Then   $\Gamma$ acts as 0 on the Albanese kernel $T(X_{\Omega})$, where $\Omega $ is a universal domain containing $k$
\end{conj}
 If $p_g(X) =q(X)=0$ then Bloch's conjecture implies $A_0(X)_0=0$. In this form  Bloch's conjecture   is known to hold for all complex surfaces not of general type and for some classes of surfaces of general type, see [PW 2]. From the results of Bloch and Srinivas in  [B-S], it follows that Bloch's conjecture for a complex surface  $X$, with $p_g(X)=q(X)=0$, is equivalent to the class $\Delta_X$ of the diagonal in $A^2(X \times X)$ having valence $v(\Delta_X)=0$, see [PW 1,  4.1].  Here , if $\Gamma$ is a correspondence in $ A^2(X \times X)$, $\Gamma$ has a valence  
$v(\Gamma)$ iff $\Gamma +v(\Gamma) \Delta_X \in \sI(X) $, where $\sI(X) \subset A^2(X \times X)$ denotes the ideal of degenerate correspondences, see [Fu, 16.1.5]. If $p \in A^2(X \times X)$ is a projector which has a valence, then $v(p)$ is either 0 or -1. Since $\Delta_X$ has always valence  -1  if $v(\Delta_X)=0$  then $\Delta_X$ has two different valences. Also, if $v(\Delta_X)=0$, then rational, algebraic, homological and numerical equivalence coincide in $A^*(X)$, see [Vois 4]. On the other hand,  if $p_g(X) \ne 0$,  then the valence of a correspondence is either unique or undefined.\par
\noindent Let $X$ be a smooth projective surface with $p_g(X) \ne 0$. A natural  question to ask is to find a relation  between conjecture 1.1 and the existence of a valence for every correspondence in $A^2(X \times X)$.\par
 \noindent In Sect. 2 we prove that  a complex K3 surface $X$ satisfies Bloch's conjecture 1.1 iff  every correspondence in $A^2(X \times X)$ has a valence(see Theorem 2.9).\par
 \noindent  In Sect.3 we consider the case of a K3 surface $X$ over $\C$  with a finite group $G$ of  automorphisms of order $n$. D.Huybrects in [Huy 1] proved that a symplectic automorphism of finite order acts as the identity on $A_0(X)$. Here  (see Theorem 3.3) we show  that,  for any finite group of automorphisms $G$ on a K3 surface,  the projector  $p =(1/n)\sum _{g \in G} \Gamma_g$ has a  valence. More precisely $v(p) =-1 $ if  $G$ consists of  symplectic automorphisms and  $v(p) = 0$ if the automorphisms are non-symplectic. Then we apply these results to compute the virtual number of coincidence for the correspondence $p$ (Corollary 3.5).\par 
\noindent Huybrechts in [Huy 2]  has introduced the notion of a constant cycle curve  on   surface $X$. In Sect. 4 we prove a motivic characterization of a constant cycle curve (Theorem 4.6) which implies a 1-1 correspondence between constant cycle curves on a K3 surface $X$, fixed by a symplectic automorphism $g$, and  constant cycle curves on the K3 surface   $Y$ obtained as a minimal desingularization of $X/g$, see Corollary 4.7. We also  prove (Theorem 4.9) that   every curve of fixed points in a correspondence which has a valence $\ne -1$ is a constant cycle curve. Examples of constant cycle curves on K3 surfaces  obtained in this ways are given in Ex 4.8 and 4.10.\par
  \noindent Finally, in Section 5, we consider the case of a smooth projective family $f : \sX \to S$  of K3 surfaces over a  smooth base $S$, a  case which is  related to the so called  generalized Franchetta's conjecture (see [O'Gr, 5.3]), where $\sX$ is the universal family of   K3 surfaces with a polarization of degree $2g-2 $ and trivial automorphism group. In Theorem 5.6  we prove that Franchetta's conjecture is equivalent to $A_0(X_{\eta})\simeq \Q$, where $X_{\eta}$ is the generic fibre of $f$. Using this result we show (see Corollary 5.7 ) that the conjecture holds true if$g=3,4,5$, in which cases the projective model in $\P^g$ of a general polarized K3 surface of genus $g$ is a complete intersection projective of $g-2$ hyper surfaces and for $g=6$(see Corollary 5.9) \par
We thank C.Weibel and K.O'Grady for useful comments on a preliminary version of this paper. 
  \section {Valences of correspondences on K3 surfaces }
If $X$ is a K3 surface over  $\C$, then it   has  a  refined Chow-K\"unneth decomposition   
$h(X)=\sum_{0 \le i\le 4}h_i(X)$ with $h_1(X)=h_3(X)=0$, because $X$ has no odd cohomology. Also  $h_2(X) =h^{alg}_2(X) +t_2(X)$, with  $\pi_2 =\pi^{alg}_2 +\pi^{tr}_2$,  $t_2(X)= (X,\pi^{tr}_2)$ and $h^{alg}_2 =(X,\pi^{alg}_2)\simeq \mathbf{L}^{\oplus \rho(X)}$. 
Therefore
\begin {equation}h(X)= \un \oplus \mathbf L^{\oplus \rho} \oplus  t_2(X) \oplus \mathbf L^2 \end{equation}
where $ \rho= \rho(X)$  is  the  rank of $NS(X)_{\Q} = (Pic X)_{\Q}$, so that  $1 \le  \rho(X)  \le 20$. We also have  
$$H^i(t_2(X))=0  \   for \   i \ne 2    \  ;   \  H^2(t_2(X)) =\pi^{tr}_2 H^2(X,\Q) =H^2_{tr}(X,\Q),$$  
 $$ dim H^2(X, \Q)) = b_2(X) =22  \ ; \  dim H^2_{tr}(X,\Q)) = 22  - \rho(X),  $$
$$A_i(t_2(X))=\pi^{tr}_2 A_i(X)=0    \  for  \    i \ne 0  \  ; \  A_0(t_2(X)) =A_0(X)_0$$  
Here $T(X) =A_0(X)_0$ because  $q(X)=0$ . We also have 
 $$ \Hom_{\sM_{rat}}(\un,t_2(X))\simeq A_0(X)_0$$
 If $X$ and $Y$ are surfaces there is a map, see \cite[7.4]{KMP}
 $$ \Psi_{X,Y} :   A^2(X \times Y) \to  \Hom_{\sM_{rat}}(t_2(X), t_2(Y))$$ 
sending $\Gamma$ to $\pi^{tr}_2(Y) \circ\Gamma \circ \pi^{tr}_2(X)$, satisfying the following functorial relation 
$$\Psi_{X,Z} (\Gamma' \circ \Gamma) = \Psi_{Y,Z}(\Gamma')\circ \Psi_{X,Y}(\Gamma),$$ 
where $X,Y,Z$ are smooth projective surfaces,  $\Gamma \in A^2(X \times Y)$ and $\Gamma' \in A^2(Y \times Z)$. In the case $X=Y$ the map $\Psi_X =\Psi_{X,X}$  yields an isomorphism of rings
 \begin {equation}\Psi_X: A^2(X \times X)/\sJ(X) \to  End_{\sM_{rat}}(t_2(X))\end{equation}
 where $\sJ(X)$ is the ideal of $A^2(X \times X)$
generated by the classes of correspondences which are not dominant
over $X$ by either the first or the second projection (see \cite[7.4.3]{KMP}). The following result shows that, if  $q(X)=0$,  then  $\sJ(X)$ coincides with the ideal $\sI(X)$ of degenerate correspondences.
\begin{lm}\label{I(X)=J(X)}
Let $X$ be a smooth projective surface with $q(X)=0$.
Then  $\sI(X)=\sJ(X)$ in $A^2(X \times X)$.
\end{lm}
\begin {proof} From the definition of the ideals $\sJ(X)$ and  $\sI(X)$ we get $\sI(X\subseteq \sJ(X)$. Let $\Gamma \in \sJ(X)$ such that $\Gamma$ 
is not dominant over $X$ under the first projection. We claim that $\Gamma$ belongs to the ideal of degenerate correspondences. 
 $\Gamma$ vanishes on some $V \times X$, with $V$ open in $X$, hence it has support on $W \times X$, with $\dim W \le 1$. If $\dim W =0$ then 
 $\Gamma = \sum_i  n_i [X \times P_i] $ in $A^2(X\times X)$, where $P_i$ are closed points in $X$. Hence $\Gamma \in \sI(X)$.
If $\dim W =1$ then  $\Gamma  \in A^1(W \times X)$, where $A^1(W \times X) = p^*_1(A^1(W)) \times p^*_2(A^1(X))$,  with $p_i$ the projections, because $H^1(X,\sO_X)=0$. Therefore  $\Gamma   \in \sI(X)$. 
\end{proof}
\begin{cor} Let $X$ be a K3 surface. Then the map $\Psi_X :  A^2(X \times X) \to \End _{\sM_{rat}}(t_2(X)$ yields the following isomorphisms
$$ A^2(X \times X)/\sI(X) \simeq \End _{\sM_{rat}}(t_2(X))\simeq A_0(X_{k(X)})/A_0(X)$$
where $ A_0(X_{k(X)}) = \varinjlim_{U\subset X} A^2 (U \times X)$ 
\end{cor}
\begin {proof} Everything follows from Lemma 2.3 and the isomorphism in [KMP.7.5.10], because $q(X)=0$. In the isomorphism  $ \End _{\sM_{rat}}(t_2(X)\simeq A_0(X_{k(X)})/A_0(X)$ the class $[\xi]$ of the generic point $\xi $ of $X$ in the ring  $A_0(X_{k(X)})$ corresponds to the identity map of the motive $t_2(X)$.
\end{proof}  

\begin {lm}\label {int-pair} Let $X$ be a complex K3 surface and let $Z \in A^2(X \times X)$ which acts as 0 on $H^{2,0}(X)$. Then $Z $ acts as 0 on $H^2_{tr}(X)$.\end{lm}
\begin {proof} Under the intersection pairing on $H^2(X,\C)$ the orthogonal complement of $NS(X) \otimes \C$ is $\pi^{tr}_2 H^2(X)$, where $\pi^{tr}_2$ is the class of $\Delta_X$ in $H^2_{tr}(X) \times H^2_{tr}(X) \subset H^4(X \times X)$. This immediately follows from the equality
$$a \cdot b =(a\cdot \pi^{tr}_2)\cdot (\pi^{alg}_2\cdot b) = a\cdot (\pi^{tr}_2 \cdot \pi^{alg}_2) \cdot b =0$$
where $a \in H^2_{tr}(X)$ and $ b\in NS(X)$, because $\pi^{tr}_2 \cdot \pi^{alg}_2 =0$. On the other hand
$$NS(X)_{\Q} =H^{1,1} \cap H^2(X,\Q) = \{x \in H^2(X) /x \cdot \omega =0 , \forall \omega \in H^{2.0}(X)\}.$$
Therefore, if $Z$ acts as 0 on $H^{2,0}(X)$ then  the orthogonal complement of  $\ker (Z)|_{H^2_{tr}}$ is contained in $NS(X) \otimes \C$, hence  $Z$ acts as 0 on $\pi^{tr}_2H^2(X)$.\end{proof}
\begin{defn} 
Let $X$ be a smooth projective variety  of dimension $d$ over a field $k$. 
The {\it indices} of a correspondence $\Gamma\subset X\times X$
are the numbers $\alpha(\Gamma)=\deg(\Gamma\cdot [P\times X])$ and
$\beta(\Gamma)=\deg(\Gamma\cdot[X\times P])$, where $P$ is any rational point on $X$; see \cite[16.1.4]{Fu}. 
The indices are additive in $\Gamma$, and $\beta(\Gamma)=\alpha({}^t\Gamma)$.\par
\noindent   A correspondence is said of {\it valence zero} if it belongs to
 the ideal $\sI(X)$ in $A^d(X \times X)$ of degenerate correspondences,
 A correspondence $\Gamma$ has {\it valence $v$} if 
$\Gamma +v\Delta_X$ has valence 0.  If $\Gamma_1, \Gamma_2$ in $A^d(X\times X)$ have valences $v_1,v_2$ then $\Gamma=\Gamma_1+\Gamma_2$ has valence $v_1 +v_2$, and $\Gamma_1\circ\Gamma_2$ has valence $-v_1v_2$ by \cite[16.1.5(a)]{Fu}. If $p\in A^d(X \times X)$is a projector, i.e. $p^2=p$, which has a valence, then $v(p)$ equals either 0 or -1.
\end{defn} 
\begin{lm}\label{Murre-D} Let $X$ be smooth projective surface over a field $k$ of characteristic 0 with  $q(X)=0$ and  $p_g(X) \ne 0$. Assume that every correspondence $\Gamma \in A^2(X \times X) $ has a valence.Then\par
(i) $A^2(X \times X)_{hom} \subset \sI(X)$;\par
(ii) Bloch's conjecture 1.1 holds true.
\end{lm} 
\begin{proof} Let  $\Gamma \in A^2(X \times X)_{hom}$, with $\Gamma +v(\Gamma)\Delta_X \in \sI(X) $. Let
   $$cl : A^2(X \times X) \to H^4(X \times X) \simeq \sum_{0 \le p +q \le 4} H^p(X) \times H^q(X).$$
Then  $cl(\Gamma) =0$ and  $v(\Gamma) cl(\Delta_X) = cl(Z)$, with $Z \in \sI(X)$. For every cycle $Z\in \sI(X)$ the component of $cl (Z)$ in $H^2_{tr}(X) \otimes H^2_{tr}(X)$ vanishes. Therefore $v(\Gamma) cl(\Delta_X) =0$ in $H^2_{tr}(X) \otimes H^2_{tr}(X)$.  Since $p_g(X) \ne 0$ we have  $cl(\Delta_X) \ne 0$, hence $v(\Gamma) =0 $, i.e. $\Gamma  \in \sI(X)$. It follows that $\Gamma \in \sI(X) \cap A^2(X \times X)_{hom}$, which proves (i).\par
We have $\Gamma= Z_1 + Z_2 $, with $Z_1 =\sum_i n_i[P_i \times X]$, $Z_2 =\sum_j m_j[X \times Qj]$ and  $\sum_i n_i =\sum_j m_j =0$.  Let $\Omega $ be a universal domain containing $k$. Then, for every $\alpha \in A_0(X_{\Omega})_0$ we have 
$(Z_1)_*(\alpha) = (Z_2)_*(\alpha) =0$. Therefore $\Gamma$ acts as 0 on $A_0(X_{\Omega})_0 =T(X_{\Omega})$ and this proves that $X$ satisfies Bloch's conjecture  1.1.
\end{proof}
\begin{rk} Note that condition (i) above is equivalent to\par
\noindent  $\End_{\sM_{rat}}(t_2(X)) \simeq  \End_{\sM_{hom}}(t_2(X))$ and implies that every endomorphism of $h(X)$ which is homologically trivial is nilpotent, see [KMP,7.6.12].\end{rk}
 \begin{thm}\label{main thm}   Let $X$ be a   complex K3 surface .Then $X$ satisfies Bloch's conjecture if and only if   every correspondence $\Gamma \in A^2(X \times X)$ has a valence.
  If  $h(X)$ is finite dimensional  then all valences belong to $\bar Q$. \end{thm}  
  \begin {proof} If every  correspondence $\Gamma \in A^2(X \times X)$ has a valence then $X$ satisfies Bloch's conjecture, by \ref{Murre-D} (ii).\par
  \noindent Conversely suppose that $X$ satisfies Bloch's conjecture and let $\Gamma \in A^2(X \times X)$. Let $\omega$ be a generator of $H^{2,0}(X)$ as a $\C$-vector space and let $\Gamma^*(\omega) =  \alpha_{\Gamma} \cdot \omega$, with $\alpha_{\Gamma} \in \C$. Let $Z = \Gamma -\alpha_{\Gamma} \Delta_X \in A^2(X \times X) \otimes \C$. Then the correspondence $Z$ acts as 0 on $H^{2,0}(X)$. From \ref{int-pair} $Z$ acts as 0 on $\pi^{tr}_2H^2(X)$.Therefore  $\bar Z=\pi^{tr}_2 \circ Z \circ\pi^{tr}_2$ acts as 0 on $H^2(X)$, hence  its cohomology class vanishes. From  Bloch' s conjecture we get that $\bar Z$ acts as 0 on $ A_0(X_{\Omega})_0$, where $\Omega=\C$. Since $\pi^{tr}_2$ acts as the identity on the group of 0-cycles of degree 0  also  $ Z$ acts as 0 on $A_0(X_{\Omega})$. Let $K =k(X)$ and choose an embedding $\sigma :  K \subset \Omega$. $\sigma$ induces an injective map $A_0(X_K) \to A_0(X_{\Omega})$. Therefore  $Z$ acts as  0  on $A_0(X_K)$. Let $[\xi]$ be the class of the generic point $\xi$ of $X$ in the quotient $A_0(X_K)/A_0(X)$. Then $Z_*([\xi])= =0$. From the isomorphism  in Corollary 2.4 we get $Z = \Gamma - \alpha_{\Gamma} \Delta_X \in \sI(X)$   and    hence $\Gamma $ has valence $v(\Gamma) =- \alpha_{\Gamma}$ in $A^2(X \times X) \otimes \C$.\par
\noindent If $h(X)$ is finite dimensional then $t_2(X)$ is evenly finite dimensional.  Let $\Gamma +v(\Gamma)\Delta_X \in \sI(X)$, with $v(\Gamma) \in \C$. Then 
$$ \Psi_X(\Gamma) = -v(\Gamma) \cdot id_{t_2(X)} \in (\End_{\sM_{rat}}(t_2(X)) \otimes C,$$
where $\Psi_X$ is the map in (2.2).  Let $n$ be such that $\bigwedge^n t_2(X)=0$. Then there exists a non-zero polynomial $G(T) \in \Q[T]$ of degree $n-1$ such that the  endomorphism $f_{\Gamma} =- v(\Gamma) \cdot id_{t_2(X)}$ satisfies $G(f_{\Gamma})=0$, see [MNP, Theorem 5.5.1]. Therefore $-v(\Gamma) \in \bar \Q$.
 \end{proof}
In [Vois 2, Cor.3.10] it is proved that, if the motive $h(X)$ of a complex K3 surface $X$ is finite dimensional, then $t_2(X)$ is indecomposable, that is any sub motive of it is either the whole motive or it is the 0-motive. The following Corollary shows that the same result holds true if every correspondence has a valence.
\begin {cor}\label {t2(X) } Let $X$ be complex K3 surface such that every correspondence in $ A^2(X \times X)$ has a valence.  Then the motive $ t_2(X)$ is  indecomposable  \end{cor}
\begin {proof}The category $\sM_{rat}$ is pseudoabelian, hence a submotive $M$ of $t_2(X)$ is defined by a projector $\pi \in A^2(X \times X)$ such that 
$\pi\circ \pi^{tr}_2=\pi^{tr}_2 \circ \pi =\pi$. $\pi$ has a valence which may be either 0 or -1.   If $v(\pi) =-1 $ then $\pi$ acts as the identity on $H^{2,0}(X)$, hence  the correspondence $Z= \pi^{tr}_2- \pi$ acts as 0. By the same argument as in the proof of  Theorem 2.9  $Z $ is homologically trivial , hence, by \ref{Murre-D}(i) $Z \in \sI(X)$.Therefore $\pi $ equals the identity of $t_2(X)$. Similarly if $v(\pi)=0 $ then $\pi$ acts is the 0-map in $\End_{\sM_{rat}}(t_2(X))$. Therefore every submodule of $M=t_2(X)$ is either the whole motive or the 0-motive. This proves that $t_2(X)$ is indecmposable.  
  \end{proof}
 \begin {rk} Let $X$ be a general complex K3 surface, i.e  $X$ is a general member of  a smooth projective families $\{X_t \}$  over the disk $\Delta$.  For a general complex K3 surface $X$ it is known that the homological motive $t^{hom}_2(X) \in \sM_{hom}$ is absolutely simple, i.e it is   simple in the category $\sM_{hom}(\bar \Q)$, see [Ped1.5.4]. If  every correspondence $\Gamma \in A^2(X \times X)$ has a valence then, by Lemma  2.6 and Remark 2.7,  
 $$End_{\sM_{hom}}(t^{hom}_2(X)) \simeq End_{\sM_{rat}}(t_2(X)).$$  
Therefore $End_{\sM_{rat}} (t_2(X)) \simeq \bar \Q $. \end{rk} 
 
   \section{ Finite group of automorphisms on K3 surfaces}

 Let $X$ be a complex K3 surface and let $g$ be an  automorphisms of $X$.The global holomorphic $2$-forms $H^{2,0}(X)$ has complex dimension 1, i.e.  $H^{2,0}(X) \simeq  \C \omega$. Since any automorphism $g \in Aut \ (X)$ preserves the vector space $H^{2,0}(X)$, there is a non-zero complex number 
$\alpha(g) \in \C^*$ such that $g_*(\omega)= \alpha(g) \omega$. 
\begin{defn}  An automorphism $g\in G$  is { \it symplectic}  if $\alpha (g) =1$, i.e if $g$ acts as the identity on $H^{2,0}(X)$, while $g$ is   non symplectic if  $\alpha(g) \ne 1$. If $g$ is non-symplectic and has order $m$ then $\alpha (g) =\zeta $, is  a primitive $m$-th root of unity in $\C^*$ \end{defn}
Nikulin (see  \cite[Thm.3.1]{Ni}) proved that $\alpha(Aut X)$ is a finite cyclic group  of order $m$ and the Euler function $\phi(m)$ divides the rank of $T_{X,\Q}$.  Here $T_{X,\Q}=H^2_{tr}(X,\Q)$ denotes the lattice of transcendental cycles on $X$.  $T_{X,\Q}$ can be described as the orthogonal complement of $NS(X)_{\Q}$ in $H^2(X,\Q)$. $H^2_{tr}(X,\Q)$ is the smallest sub-Hodge structure of $H^2(X,\Q)$ containing $H^{2,0}( X)$.Since $H^{2,0}(X) \subset H^2_{tr}(X ) \otimes \C$ is compatible with the action of an automorphism $g$ of $X$, $g$ is symplectic iff it acts as the identity on $H^2_{tr}(X)$.\par 
 \noindent   Let $g$ be   of order $m$ and let $p =(1/m) \sum_{0\le i\le m-1} \Gamma_{g^i}$ and   $\Psi_X(p) =\pi \in End_{\sMrat}(t_2(X))$. $\pi$ is a projector, hence $\pi$  acts either  as  0 or as the identity on $H^{2,0}(X) \simeq \C \omega$. If $g$ is  non-symplectic then $g_*(\omega) = \zeta  \omega$, with $\zeta $ a primitive $m$-th root of 1, so that $\pi$ acts as 0 on $H^{2,0}(X)$.  If $g$ is symplectic then $g_*(\omega)= \omega$ and $\pi$ acts as the identity on $H^{2,0}(X)$.  In the first case   $\pi_*= 0$ on $H^2_{tr}(X,\Q)$.  In the second case  $\pi_* $ equals the identity  on $H^2_{tr}(S,\Q)$
 The following result has been proved by C.Voisin (see [Vois 1] ) in the case of a symplectic involution, and then extended by D.Huybrects in [Huy 1]  to any symplectic automorphism  $g$ of finite order.  
 \begin {thm}(Huybrechts) 
Let $g$ be a symplectic automorphism of finite order on a complex K3
 surface $X$.Then $g$ acts as the identity on $A_0(X)_0$.\end{thm} 
A finite group $G$ of automorphism of the K3 surface $X$ acts on the transcendental motive $t_2(X)$, via the action of the correspondences $\Gamma_g$.  The following result characterizes the action of symplectic and non-symplectic automorphisms on $t_2(X)$and proves that the projector $p =(1/n)\sum _{g \in G} \Gamma_g$ has a valence.
\begin{thm}\label{Theorem 2} 
Let $X$ be a K3 surface over $\C$  with a finite group $G$ of order $n$ of
automorphisms.  Let $Y$ be a  minimal desingularization of the quotient surface $X/G$.
Then \par
   (i) If  $G$ consists of symplectic automorphisms $t_2(X)^G =t_2(X)= t_2(Y) $ and the rational map $X \to Y$ induces an isomorphism of motives $h(X) \simeq h(Y)$.\par
   (ii)  If  $G$ consists of  non- symplectic automorphisms, then   $t_2(X)^G =t_2(Y)=O$, which implies that $p_g(Y)=0$.\par
   \noindent The projector  
$$p =(1/n)\sum _{g \in G} \Gamma_g$$
has valence $v(p)= -1$,  if $G$ consists of symplectic automorphisms, while $v(p)=0 $ if the automorphisms of $G$ are non symplectic.
 \end{thm}
\begin {proof}  (i) From Theorem 3.2   every symplectic automorphism 
$g \in G$ acts trivially on $A_0(X)$, so that $A_0(X)^G =A_0(X)$.
 Since $G$ is symplectic there are a finite number
 $\{P_1,\cdots,P_k\}$ of isolated fixed points for $G$ on $X$. Let $Y$
 be a minimal desingularization of the quotient surface $X/G$. The
 maps $f : X \to X/G$ and $\pi : Y \to X/G$ yield a rational map $X
 \to Y$, which is defined outside $\{P_1,\cdots,P_k\}$.  
 $t_2(-)$ is a birational invariant for smooth projective surfaces, hence we
 may blow up $X/G$ to assume that $Y =X/G$. We get a map $\theta :
 t_2(X) \to t_2(Y)$, and we claim that $\theta$ is the projection onto
 a direct summand. Let $f : X \to Y$. Then $\theta =
 \Psi_{X,Y}(\Gamma_f)$.  $\Gamma_f$ has a right inverse $(1/m)({}^t\Gamma_ f)$ and $m
 p={}^t\Gamma_ f \circ \Gamma_f =\Delta_X +\sum_{g\ne 1} \Gamma_g$. It
 follows, from the functoriality of the map $\Psi$,  that $t_2(X)^G$ is
 a direct summand of $t_2(X)$. Let
$$t_2(X)=t_2(X)^G \oplus N.$$
Since $A_i(t_2(X))=0$, for $ i \ne 0$, and $A_0(X)^G =A_0(X)$ we get $A_i(N)=0$, for all $ i$. From \cite[Lemma 1]{GG} we get  
$N=0$, hence  $t_2(X) =t_2(X)^G \simeq t_2(Y)$.\par
The surface  $Y$ is a K3 surface and $H^2_{tr}(X,\Q) \simeq H^2_{tr}(Y,\Q)$ because every $g \in G$ acts trivially on $H^2_{tr}(X)$. Thus $\rho= \rho(X) =\rho(Y)$. The motives $h(X)$ and $h(Y)$ have Chow-K\"unneth decomposition as follows
 $$h(X)\simeq \un \oplus \L^{\oplus  \rho} \oplus t_2(X) \oplus \L^2;$$
 $$h(Y)\simeq \un \oplus  \L^{\oplus  \rho} \oplus t_2(Y \oplus  \L^2.$$
from the isomorphism $t_2(X) \simeq t_2(Y)$ we get $h(X) \simeq h(Y)$.\par
  (ii)  Let $Y$ be a desingularization of the quotient surface $X/G$. We have $t_2(Y) =t_2(X)^G$,  because $t_2(X)$ is a birational invariant.  Also $q(Y) =0$   and $p_g(Y)=0$, see \cite[Thm.3.1a)]{Ni}.  $Y$ is a not a surface of general type, hence  $A_0(X)_0 =0$ which implies $t_2(Y)=0$.\par
\noindent If every $g \in G$ is symplectic then , by (i), $\Gamma_g$ acts as the identity on $t_2(X)$, for all $g \in G$. From the isomorphism in Cor 2.4  we get $\Gamma_g -\Delta_X \in \sI(X)$, for every $g \in G$ i.e $v(\Gamma_g)= -1$. Therefore $v(p) =-1$. If $G$ consists of  
non-symplectic automorphisms, then, by (ii), $t_2(Y) =t_2(X)^G=0$, and hence $p$ acts as 0 on $t_2(X)$. Therefore $p \in \sI(X)$ and $v(p)=0$. 
 \end{proof}
Next we   apply   Theorem 3.3 to compute the virtual number of coincidences for the projector $p$ associated to a finite group of automorphisms on a  K3 surface, using  a formula by Severi , as reconstructed in [PW 1,p 499].
 \begin{thm}[Severi's formula]\label{Theorem 3}
Let $X$ be a smooth projective surface. If $\Delta_X$ does not have
valence 0 and $\Gamma\in A^2(X\times X)$ is a correspondence with
valence $v$, then
$$
\deg(\Gamma\cdot\Delta_X) = 
\alpha(\Gamma) +\beta(\Gamma) + t(\Gamma)
 +  v\cdot\left(2+\rho(X)-c_2(X)\right).$$
where $t(\Gamma)$ is the trace of the action of $\Gamma$ on the 
finite dimensional vector space $NS(X) \otimes\Q$.
 \end{thm}

 \begin{cor}\label{p and q} 
Let $X$ be a K3 surface with a finite
group $G$ of  automorphism of order $n$. Let $p =(1/n)\sum_{g \in G} \Gamma_g$. If $G$ consists of symplectic automorphisms then  
\begin{equation}
\deg(p \cdot\Delta_X) = 2+ t(p)+22-\rho(X)=24 +t(p) -\rho(X) \end{equation}
while, if the automorphisms of $G$ are non-symplectic, 
\begin{equation} \deg(p \cdot\Delta_X)= 2 +t(p)\end{equation}
 where $t(p)$ is the trace of the action of $p$ on the 
finite dimensional vector space $NS(X) \otimes\Q$.
\end{cor}
\begin {proof}
From Theorem 3.3  the projector $p$ has valence $v(p) =-1$ in the case $G$ consists of symplectic automorphisms while  $v(p)=0$ in the non-symplectic case.
Therefore everything follows from   Theorem \ref{Theorem 3}, because
$c_2(X)=24$ and the correspondence $p$ has  indices 1.
\end{proof}
\begin{ex} The following example shows how the equality in (3.7) can be used to recover some results, obtained by Nikulin in [Ni], for symplectic automorphisms of prime order on a K3 surface.\par
\noindent Let $X$ be a K3 surface with a symplectic automorphism $g$ of order a
prime number $l$.  Let $G$   be the group  of order $l$ generated  by $g$ and
 let $t(g)$ be the trace of the action of $g$ on the finite
dimensional vector space $NS(X)_{\Q}$. From Theorem 3.3
$g$ acts as the identity on $t_2(X)$ hence   the correspondence $\Gamma_g \in A^2(X
\times X)$ has valence -1.  Let $k$ be the number of isolated fixed
points of $g$.  All the fixed points $x \in X$, under the action of $G$, have the same order $l$,
because the stationary subgroup  $G_x$ of $x$, being  a cyclic subgroup of $G$, coincides with $G$. Therefore from (3.7) we get
\begin{equation} k = deg (\Gamma_g \cdot  \Delta_X) =  24 +t(g) - \rho(X) \end{equation}
i.e.   $t(g) = k +\rho(X) - 24$, for every $g \in G$, $g \ne 1$, where $t(g) =tr_{NS(X)}(g)$.  \par
 \noindent  Let    $p =(1/l ) \sum_{g \in G} \Gamma_{g}$ the projector in $A^2(X \times X)$. Since all automorphisms $g^i \in G$, with $i \ne 0$,  have the same number of fixed points we get
 $$ t(p)= tr_{NS(X)}(p ) =1/l(\rho(X) +(l-1)t(g)).$$
Let $M =(X,p) \in \sM_{rat}$ : then $h(X) = M \oplus N$, with $N =(X,1-p)$ and 
$$A^1(h(X)) = NS(X)_{\Q}= A^1(M) \oplus A^1(N).$$
where   $A^1(M)= p^*A^1(X)\simeq NS(X)^G_{\Q}$ and   $A^1(N)=(1-p)^* A^1(X)$. The projector $p$ acts as 0 on $A^1(N)$,hence
$$t(p) = tr_{NS(X)}(p )= tr_{NS(X)^G_{\Q}}(p) =s =1/l(\rho(X) + (l-1)t(g))$$
where $s =\dim_{\Q}(NS(X)^G_{\Q}$. Therefore 
\begin{equation} t(g) = {ls-\rho(X) \over l-1}\end{equation}
Here $s \ge 1$ because every automorphism  fixes a hyperplane section of $X$.\par 
\noindent Let $Y$ be a minimal resolution of the singular points  of $X/G$. Then $Y$ is a K3 surface and, by Theorem 3.3 ,  $h(X) \simeq h(Y)$. In particular 
$$A^1(h(X)) \simeq A^1(h(Y)) \simeq \L^{\oplus \rho(X)}.$$
 Let $\pi$ be the rational map obtained by composing the morphisms $X \to X/G$ and $Y \to X$.  $\pi$ is defined outside of a finite number of points in $X$. Therefore $\pi$ induces a map $NS(X)_{\Q} \to NS(Y)_{\Q}$ where  
$$ NS(X)_{\Q} = NS(X)_{\Q}^G  \oplus A^1(N)  \ ;  \  NS(Y)_{\Q} = NS(X)_{\Q}^G \oplus M_Y.$$
\noindent Here  $M_Y$ is the subvector space of $NS(Y)_{\Q}$ generated by the classes   of  the divisors  corresponding to  the curves obtained by resolving the singularities of $X/G$.  All the fixed points of $G$ on $X$ have order  $l$ and, for every  fixed point $x$,  the minimal resolution gives a curve consisting of $(l-1)$ rational curves.  Therefore $M_Y$ has dimension $k (l-1)$and we get
\begin{equation}\rho(X) =\rho(Y)= dim_{\Q} (NS(X)^G_{\Q}  + k(l-1) = s +k(l-1) \end{equation}
The equalities in (3.9), (3.10) and (3.11)     give 
$$ k = 24/l+1$$
We also have  $l \le 7$ because from (3.11)   $l=11$ implies  $\rho(X) \ge 21$ which is impossible. Note that, by a result of Nikulin  in [NI], the order of a symplectic automorphism is always $\le 8$. In [GS,Th.0.1] it is proved that a K3 surface cannot have a symplectic and a non- symplectic automorphism of the same order $l =7$.
  \end{ex}

 \section {Constant cycle curves on K3 surfaces}
  Let $X$  and $Y$ be  a smooth projective varieties over a field $k$  and let $h(X)$ and $h(Y)$ be their motives in $\sM_{rat}(k)$.
 If we pass to the (covariant) category $\sM^o_{rat}(k)$ of
birational motives ( see [KMP, 14,7,5]) the Lefschetz motive $\L$
goes to zero.  Writing  $\bar h(X)$ for the birational motive of a smooth
projective variety $X$ of dimension $d$, and $\Hom_{bir}$ for morphisms
in $\sM^o_{rat}(k)$, the fundamental formula is:
$$A_0(X_{k(Y)})=\varinjlim_{U\subset Y} A^d (U \times X)
                 \cong \Hom_{bir}(\bar h(Y), \bar h(X)).$$
 In particular, $A_0(X)\cong \Hom_{bir}(\un,\bar h(X))$ for all $X$. Let  $X$ be  a smooth projective surface, $C$ a smooth
projective curve. Let $h(C)=\un \oplus h_1(C) \oplus \L$ and 
 $h(X) =\un \oplus h_1(X) \oplus  h^{alg}_2(X)\oplus t_2(X) \oplus h_3(X)
\oplus \L^2 $ be   Chow-K\"unneth decomposition for the motives
of $X$and $C$.   Since the birational   motive of a curve is  $\un\oplus \bar h_1(C)$ and the birational motive of a surface  $X$ is $\un + \bar h_1(X) + \bar t_2(X)$
 we get 
 $$A_0(X_{k(C)}) \simeq A_0(X) \oplus \Hom_{bir}(\bar h_1(C), \bar h_1(X)) \oplus \Hom_{bir}(\bar h_1(C), \bar t_2(X))$$
 If $q(X)=0$ then $h_1(X)=h_3(X)=0$ . Moreover $h_1(C)$ and $t_2(X)$ are birational motives, see[KMP,148.8]. Therefore
\begin {equation}A_0(X_{k(C)})\cong  A_0(X) \oplus  \Hom_{\sM_{rat}}(h_1(C),t_2(X))\end{equation} 
\begin {rk} The isomorphism in (4.1) depends only on $k(C)$, hence it does not change if we replace $C$ with any other curve which is a smooth projective model of $k(C)$.\end{rk}  
 Let $X$ be a smooth projective surface with $q(X)=0$ and let  $C$ be a  smooth curve.   Consider the map
 $$A_1(C \times X) \to  \varinjlim_{U\subset C} A_1(U \times X)\simeq A_0(X_{k(C)}),$$
where $k(C) = k(\eta_C)$, with $\eta_C$  the generic point of $C$ and let
$$\Psi_{C,X} : A_1(C\times X) \to  {A_0(X_{k(C)}) \over A_0(X)} \simeq   \Hom_{\sM_{rat}}(h_1(C),t_2(X))$$
be the composite map .
 \begin{lm}\label{lem:Psi} Let $\sI(C,X)$ denote the subgroup of $A_1(C\times X)$ generated by
the classes of degenerate correspondences, i.e.  correspondences of the form $P\times D$ and $C\times Q$,
where $P\in C$ and $Q\in X$ are closed points and $D$ is a
curve on $X$. 

The map $\Psi_{C,X}$ induces an isomorphism
$${A_1(C\times X) \over \sI(C,X)} \simeq Hom_{\sM_{rat}}(h_1(C),t_2(X)).$$
  \end{lm}
\begin{proof} The homology class of every $Z \in A_1(C \times X)$ lies in $ H_2(C) \oplus H_2(X)_{alg}$, so that adding to $Z$ vertical and horizontal cycles,which belong to $\sI(C,X)$,  we get a cycle $Z' $ homologous to zero  such that $Z= Z ' \in A_1(C\times X)/\sI(C,X)$ . So we  may assume that $Z \in A_1(C\times X)_{hom}$. The correspondences $P\times D$ generate the kernel of $A_1(C\times X)\to A_0(X_{k(C)})$.  The correspondence $C\times Q$
maps to the image of $Q$ under $A_0(X)\to A_0(X_{k(C)})$. Therefore
$\sI(C,X)$ is in the kernel of the composition
$$
\Psi_{C,X}: A_1(C\times X) \to A_0(X_{k(C)}) 
\to Hom_{\sM_{rat}}(h_1(C),t_2(X))
$$
Conversely, let $Z \in A_1(C \times X)_{hom}$  such that its 
class $\Psi_{C,X}(Z)$ in $A_0(X_{k(C)})$ is in the subgroup $A_0(X)$.
Then there are closed points $Q_i$ of $X$ such that
$\Psi_{C,X}(Z)=\sum n_i \Psi_{C,X}(C\times Q_i)$. 
Subtracting $\sum n_i[C\times Q_i]$ from $Z$, we may assume 
that $\Psi_{C,X}(Z)=0$. But then $Z$ 
vanishes in $A_1(U\times X)$ for some $U=C-\{P_j\}$ and hence
$Z$ is a linear combination of cycles of the form $P_j\times D_j$.
\end{proof}
The following definition has been given in [Huy 2]
\begin {defn} 
Let  $X$ be a smooth projective surface over an algebraically closed field $k$ ,  $C$ an integral  curve on $X$ and let $f : C \to X$ be a closed immersion.  $C$ is called a {\it constant cycle curve} if  the class $[\eta_C]$ belongs to the image of $f_*:  A_0(X) \to A_0(X_{k(C)})$. Here  $\eta_C$ is the generic point of $C$ and $[\eta_C]$ is viewed as a closed point of the surface $X_{k(C)}$ over the field $k(C)$.
 \end{defn}
 \begin {rk} Every rational curve $C$ is a constant cycle curve. In the above definition, by eventually taking $k( \tilde C)= k(C)$, where $\tilde C$ is a desingularization of $C$, we  can assume $C $ to be smooth. If $X$ is defined over $\C$, then by [Huy 2, 3.7] a curve $f : C \to X$  is a constant cycle curve iff  the induced map $f_* : A_0(C)_0 \to A_0(X)_0$
vanishes.\par
\noindent  If $X$ is  a K3 surface over $\C$ then a  curve $C$ is a constant cycle curve iff any point on $C$ is rationally equivalent to $c_X$, see [Vois 3, 2.2].\par
\noindent For every K3 surface $X$ over an algebraically closed field $k$,  $c_X$  is  the distinguished class of degree one in $ A_0(X)$,  introduced in [BV]. $c_X$ is the class  of any closed point $P \in X$ lying on rational curve $R$.  Because for any irreducible curve $C$ on $X$ there is a rational curve $R \ne C$ which intersect $C$, we can represent $c_X$ by the class of a point $P \in C$, namely any point of $C \cap R$. 
\end{rk}
The following result gives a condition on the motives of $C$ and $X$ in order  for $C$ to be constant cycle curve, in the case of a surface  $X$ with $q(X)=0$. 
\begin {thm}\label{Theorem 4} 
Let $X$ be a smooth projective surface over a field $k$,
with  $q(X)=0$, and let $f : C \to X$  be a   curve on $X$. The following conditions are equivalent\par
 (i) $C$ is a constant cycle curve;\par
 (ii) The map $f : C \to X$ induces the 0- map $f_*$ in 
$Hom_{\sM_{rat}}(h_1(C),t_2(X))$.
\end{thm}
\begin {proof} From the commutative diagram 
$$ \CD
 A_1(C \times C)@>{f_*}>>A_1(C \times X)  \\
 @VVV   @V{\Psi_{C,X}}VV    \\
{ A_0(C_{\eta_C})\over A_0(C)} @>{\bar f_*}>>{A_0(X_{k(C)})\over A_0(X)} \\
    \endCD$$
where $\bar f_*([\eta_C]) = [\eta_C] \in A_0(X_{k(C)})$, we get $\Psi_{C,X}(f_*(\Delta_C)) =[\eta_C] $. Therefore $[\eta_C] =0$ in ${A_0(X_{k(C)})\over A_0(X)}$
iff $f_*(\Delta_C) \in \sI(C,X)$.  From Lemma 4.3 , it follows that $f_*(\Delta_C) \in \sI(C,X)$ iff 
the map induced by $f$ in $\Hom_{\sM_{rat}}(h_1(C), t_2(X))$ vanishes. Therefore $(i) \Leftrightarrow (ii)$.\par
\end{proof}
 \begin {cor} Let $X$ be a complex K3 surface with a symplectic automorphism $g$ of finite order $n$. Let $Y$ be the minimal desingularization of $X/g$. Then  there is i a 1-1 correspondence between constant cycle curves on $X$ fixed by $g$ and constant cycle curves on   $Y$.\end{cor}
 \begin{proof} Let $f : C \to X$ be a constant cycle curve  on $X$. The maps $\pi: X \to X/g$  and $ Y \to X/g$ yield a rational map $X \to Y$, which is defined outside a finite number of points on $X$. $t_2(X)$ and $h_1(C)$ are birational invariants, hence we may blow up $X/g$ to assume that  $X/g =Y$ and $\pi : X \to Y$.  Then $Y$ is a K3 surface and the finite map $\pi$ induces an isomorphism between $t_2(X) $ and $t_2(Y)$, as in the proof of  Theorem   3.3. From Theorem  4.6 the map $f_* \in \Hom_{\sM_{rat}}(h_1(C),t_2(X))$ vanishes, hence also $ \tilde f _* =0$, with $\tilde f = \pi \circ f $.  Therefore  $\tilde f : C \to Y$ is a constant cycle curve.\par
 \noindent   Conversely let $D\subset Y$ be a constant cycle curve which is the image of a curve $C$ on $X$ fixed by $g$.  
The map $\pi$ induces  group homomorphisms  
$$ \pi_* :   A_0(X_{k(C)}) /A_0(X)  \to  A_0(Y_{k(C)})/A_0(Y);$$
$$ \pi^* :   A_0(Y_{k(D)}) /A_0(Y)  \to  A_0(X_{k(C)})/A_0(X).$$
Let $G$ be  he cyclic group generated by $g$. Then  $(\pi^* \circ \pi_*)(\alpha) =  \sum_{g \in G} g_*(\alpha)$ for every  $\alpha \in A_0(X_{k(C)}) /A_0(X)$.  From Corollary 2.4 and Theorem 3.3 we get the following isomorphisms   
$$ A_0(X_{k(X})/ A_0(X)  \simeq  \End_{\sM_{rat}}(t_2(X))\simeq \End_{\sM_{rat}}(t_2(Y)) \simeq A_0(Y_{k(Y)}/ A_0(Y)  $$
 where the class of the generic point $\xi$  of $X$, which corresponds to the identity map of $t_2(X)$, is mapped to the   class  $[\zeta]$ of the generic point $\zeta$ of $Y$. 
 Therefore we get  a commutative diagram 
 $$ \CD {A_0(X_{k(X})\over A_0(X)}@>{\simeq}>>{A_0(Y_{k(Y)}\over A_0(Y) }\\
  @VVV                      @VVV    \\
  {A_0(X_{k(C)}) \over A_0(X)}@>{\pi_*}>> {A_0(Y_{k(D)}) \over A_0(D)}  \endCD $$
 where the vertical maps are induced by  the specialization maps\par
 \noindent   $A_0(X_{k(X)}) \to A_0(X_{k(C)})$ and  $A_0(Y_{k(Y)}) \to A_0(X_{k(D)})$, which send  $[\xi]$ to $[\eta_C]$ and $[\zeta$ to $[\eta_D]$. Therefore $\pi_*([\eta_C] =[\eta_D]$.
 As $g$ acts as the identity on $t_2(X)$,  we get $g_*([\xi]) =[\xi] $. The specialization map  sends $[\xi] $ to $[\eta_C]$. Therefore $g_*([\eta_C] = [\eta_C]$ in $A_0(X_{k(C)})/A_0(X)$, and this  implies 
$$(\pi^* \circ \pi_*)([\eta_C]) = \sum_{g \in G} g_*([\eta_C])=n [\eta_C] \in   A_0(X_{k(C)})/ A_0(X).$$
 $D$ being a constant cycle curve on $Y$, $[\eta_D]$ is in the image of the map $A_0(Y) \to A_0(Y_{k(D)})$, and hence $\pi_*([\eta_C] )=[\eta_D] =0$ in $ A_0(Y_{k(D)})/ A_0(D)  $.
  It follows that 
  $$n[\eta_C] =  (\pi^* \circ \pi_*)([\eta_C]) =\pi^*([\eta_D]) = 0$$
  in $  A_0(X_{k(C)})/ A_0(X)$. Therefore  $C$ is a constant cycle curve. 
 \end{proof}
\begin{ex} Let $X$ be a  complex K3 surface, with $\rho(X)=9$, which is the intersection of 3 quadrics in $\P^5$,  having a  symplectic involution $\sigma$, as described in [VG-S,3.5]. The desingularization $Y$ of $X/\sigma$ is a quartic surface in $\P^3$, and $t_2(X) =t_2(Y)$, by Theorem 3.3.  Let $\beta: \tilde X \to X$ be the blow up at the 8 fixed points of $\sigma$ and let $C =(1/2) \sum_{1 \le i \le 8} C_i$  be the corresponding divisor on $Y$, with $C_i$ a rational curve. Let \par
\noindent  $f : \tilde X \to Y$ and   let $L \in NS(X)$, with $L^2 =8$, be the line bundle  that gives the embedding $\Phi_L:  X \to \P^5$. There is a line bundle $M \in NS(Y)$ such that $\beta^*L =f^*M$, with $M^2=4$ and $h^0(M-C)=2$. Note that $ NS(Y)\simeq NS(X)^{\sigma} \oplus \{[C_1],\cdots  [C_8]\}$, where $NS(X)^{\sigma}$ has rank  1, because $\rho(X) =\rho(Y) = 9$. Let $ \mu = \Phi_{M-C} : Y \to \P^1$ be the map associated to the pencil $| M-C |$. $Y \to \P^1 $ is an elliptic fibration, because $(M-C)^2=0$ implies that all  curves in the linear system have genus 1. Let $C_0 \subset Y$ be the 0-section  of $\mu$ and let $C_n$ be the closure of the set of $n$-torsion points on the smooth fibres $Y_t$, $t \in \P^1$. Then, by the results in [Huy 2, 6.2],  $C_n$ is a constant cycle curve. From Corollary 4.7  every $C_n$ pullbacks to a constant cycle curve on $X$. \end{ex}
 \noindent In [Huy 2,7.1] it is proved that, for a complex K3 surface $X$, every fixed curve of a non -symplectic automorphism $g$ of finite order is a constant cycle curve.  The following theorem extends Huybrecht's result to every correspondence $ \Gamma \in A^2(X \times X)$  on a  K3 surface  $X$,  such that $\Gamma$ has a valence $ \ne -1$.
\begin {thm} Let $X$ be a  complex  K3 surface and let  $\Gamma  \in A^2(X \times X)$. Suppose that $\Gamma$ has a valence $v(\Gamma)$, with $v(\Gamma) \ne -1$.  If $f : C \to X$
is a curve on $X$ such that $\Gamma_*([x])=[x]$ in $A_0(X)$, for every $x \in C$, then $C$ is a constant cycle curve.
\end{thm}
\begin {proof} From the isomorphism $\Hom_{\sM_{rat}}(\un,t_2(X)) \simeq A_0(X)_0$ it follows that every  0-cycle   $\alpha$ of degree 0 corresponds to a map $t_{\alpha}: \un \to t_2(X)$. The action of the correspondence $\Gamma$ on $\alpha \in A_0(X)_0$  coincides with the composition map $\bar \Gamma \circ t_{\alpha} : \un \to t_2(X) \to t_2(X)$, where
$$ \bar \Gamma =\Psi_X(\Gamma ) \in \End_{\sM_{rat}}(t_2(X) ) \simeq  {A_0(X_{k(X)}) \over  A_0(X)} .$$
The map $\bar \Gamma$ corresponds, in the isomorphism above, to the cycle class $\Gamma([\xi])= -v(\Gamma)[\xi] $, because $\Gamma +v(\Gamma)\Delta_X \in \sI(X)$. Here $\xi$ is the generic point of $X$ and the class  $[\xi]$ corresponds to the identity map of $t_2(X)$. Therefore $\bar \Gamma \circ t_{\alpha}$ is multiplication by $ -v(\Gamma)$ and 
$\Gamma_*(\alpha) = -v(\Gamma) \alpha$ in  $ A_0(X)_0$. Now let $x$ be any point on $C$ and let $y \in C$ be such that $[y] =c_X$ in $A_0(X)$. Then $\Gamma_*([x]) =[x]$ and $\Gamma_*([y]) = [y]=c_X$. Therefore $\Gamma_*(\alpha) =\alpha$,
with $\alpha =[x] -c_X \in A_0(X)_0$. We get $( v(\Gamma)+1)\alpha=0$ and hence $\alpha =0$, because $v(\Gamma)\ne -1$. This proves that, for every point $x \in C$ the 0-cycle $[x] -c_X$ vanishes in $A_0(X)_0$, i.e $C$ is a constant cycle curve.
 \end{proof}
\begin {ex}  (1) The result In Theorem 4.8  applies to the case of the fixed locus  of the {\it bitangent correspondence}, considered in  [Huy 2]. Let $X \subset \P^3$ be a smooth quartic not containing  a line. Let $x$ a generic point on $X$ and let $C_x =T_x X \cap X$. There are six lines passing trough $x$ and such  that they are bitangent to  $C_x $ at some other points
 $y_{x,1},\cdots y_{x,6} \in C_x$. The points $\{y_{x,1}\cdots y_{x,6} \}$ are rationally equivalent in $A_0(X)$. Let $\Gamma \in A^2(X \times X)$  be  the  bitangent correspondence 
 $x \to \{y_{x,1}\cdots y_{x,6} \} $. Then $\Gamma \circ \Gamma =\Delta_X$, because if $l_x$ is a bitangent trough $x$ and $y_x$ then $l_x=l_y$ is also the bitangent trough $y=y_x$ and $x_y =x$.  
Therefore $\Gamma_*$ acts as involution on $A_0(X)$. By [Huy 2,8.1] $\Gamma_*$ acts as -1 on $A_0(X)_0$, and hence  $\Gamma_*([\xi]) =- [\xi]$ in $A_0(X_{k(X)})/A_0(X) \simeq \End_{\sM_{rat}} (t_2( X))$. Therefore $\Gamma + \Delta_X \in \sI(X)$, i.e.  the correspondence $\Gamma$ has  a valence $v(\Gamma) =1$. From theorem 4.8 the curve $C$ of contact points of hyperflexes, which is  the fixed locus of $\Gamma$, is a constant cycle curve.\par
\noindent (2) Let $\sigma $ be an involution on a K3 surface  $X$ such that  $t_2(Y) =0$, where $Y $ is the desingularization of $X/\sigma$. Then, by  [Ped 2, Theorem 1],  $v(\Gamma_{\sigma}) =1$,  and hence every curve in the fixed locus of $\sigma $ is a constant cycle curve.  In the case $X$ is a double cover of $\P^2$ branched over a sextic curve $C \subset X$, then $C$ is a constant cycle curve,because $t_2(\P^2) =0$. If   $X$  is a double cover of a quadric $Q \simeq \P^1 \times \P^1$ in $\P^3$, as in [VG-S, 3.5], then  double cover is ramified on a curve $B$ of bidegree$(4,4)$ and the branch curve of the covering involution $\sigma$ on $X$ has genus 9. $C$ is a constant cycle curve,  because $t_2(Q)=0$ implies that $v(\Gamma_{\sigma})=1$.
\end{ex}
  
   \section {Families of K3 surfaces and the generalized  Franchetta's conjecture}
 Franchetta's conjecture on  line bundles over  the universal  curve  on the moduli space of curves of genus $g$, which has been proved in [AC], may be stated as follows.
 \begin {conj} (Franchetta's conjecture) Let $\sC$ be the universal curve over the function field of  the moduli space $\sM_g$ of curves of genus $g$. Then any line bundle on $\sC$ is the integral multiple of the canonical bundle,i.e.
 \begin {equation} \Pic (\sC/ \sM_g) \simeq \Z\omega_{\sC} / \sM_g \end{equation}  
 \end{conj} 
  \noindent  O.Grady in  [O'Gr,5.3] has asked the following question, which is similar to Franchetta's conjecture,  for the universal family of K3 surfaces.
 \begin{ques}(Generalized Franchetta's conjecture)  Let $g\ge3$ and let $ \sK_g$ be the moduli space of complex  K3 surfaces with a polarization of degree $(2g-2)$.  
 By restricting to the open  dense subset $S_g =\sK^0_g$ parametrising polarized K3 surfaces with trivial automorphism group,  we may assume that the  family  $f : \sX_g \to S_g $ is smooth.  Let $\sZ \in \CH^2(\sX)$. Is it true that $\sZ \vert X_s \in \Z c_{X_s}$ for all closed points $s \in S_g$?
\end{ques}
 
 \noindent   Let $f : \sX \to S $ be a   projective family of  surfaces over an algebraically closed field $k$ of characteristic 0,   i.e.  $f$ is  projective, and the fibres are  surfaces. Let $S$ be smooth of dimension $n$  and let  $U =S -Y$, with $Y$ closed in  $S$. Let  $i : Y \to S$ and  $j :  U \to S$  be the inclusions.  Let $ \bar \sX  = \sX \times_S Y$, $\sX_U = \sX \times _S U$. Then there is  an exact sequence
\begin{equation}  \CD \CH_n(\bar \sX) @>{i_*}>>\CH^2 ( \sX )@>{j^*}>>\CH^2( \sX_U )@>>>O \\ \endCD \end{equation}
where $i_*$ and $j^*$ denote  push-forward and pull-back induced by the inclusions $i$ and $j$, see [Fu, 20.3].  
\noindent  There also is a Gysin homomorphism  
$$i^{!} : \CH^2(  \sX) \to \CH^2(\bar \sX)$$
and a specialization map $\sigma: \CH^2(\sX_U)\to \CH^2(\bar \sX)$, such that $\sigma(j^*(\alpha)) = i^{!}(\alpha)$ for all $\alpha\in A^2(\sX)$. 
 Let $X_{\eta}$ be the generic fibre of $f$. Then
 \begin{equation} \lim_{U \subset S} CH^2(\sX_U)= \lim_{U \subset S} \CH^2(\sX \times_S U) \simeq \CH^2(\sX_{k(S)})=\CH^2(X_{\eta})\end{equation} 
where $k(S)$ is the quotient field of $S$ and $\eta$ the generic point of $S$.\par
\noindent The following result gives a condition equivalent to the  generalized Franchetta's conjecture. 

 \begin{thm} Let  $f : \sX \to S$ be  a smooth projective family of K3 surfaces over an algebraically closed field $k$ of characteristic 0, with $S$ smooth of dimension $n$. Let $\eta$ be the generic point of $S$, $X_{\eta}$ 
 the generic fiber  of $f$ and let $X_s =f^{-1}(s)$,  for every closed point  $s \in S$. Then the following conditions are equivalent.\par
 (i) $\CH_0(X_{\eta})_{\Q} = A_0( X_{\eta})\simeq \Q$.\par
(ii) There is a distinguished cycle $\sC \in A^2(\sX)=\CH^2(\sX)_{\Q}$ such that, for every $\sZ \in A^2(\sX)$,  $\sZ \in \Q[\sC]$, modulo vertical cycles, and
$$\sZ \vert X_s \in \Q[c_{X_s}]$$
for every closed fiber $X_s$.
 \end{thm}
 \begin{proof}  (i) $\Longrightarrow$ (ii). By performing a base change we may assume that there exists a section $\pi : S \to \sX$ of $f$ such that $\pi(s)$ represents $c_{X_s}$ in $A^2(X_s)$, for every $s \in S$.
 Let $\sC$ be the class in $A^2(\sX)$ of $(p_2)_*(\Gamma_{\pi})$, where $\Gamma_{\pi} \subset A_{dim S} ( S \times \sX)$ and $p_2 : S \times \sX \to \sX$. For every closed point $s \in S$ the exact sequence in (5.4), with $Y =\{s\}$  and $U =S -Y$ yields
 $$ \CD A_n(X_s)@>{i_*}>>A^2 ( \sX )@>{j^*}>>A^2( \sX_U)@>>>O \\ \endCD $$ 
The specialization map  $\sigma : A^2(X_U)\to A^2(\bar \sX)$  induces \par
\noindent $\sigma:  A^2(X_{\eta} )\to A^2(X_s)$. Then $i^!(\sC) =c_{X_s} \in A^2(X_s)$. Let $\alpha$ be a generator of $A_0(X_{\eta})$ and let $ a \cdot \alpha$, with $ a\in \Q$,  be the class of $j^*(\sC)$ in $A_0(X_{\eta})$ under the map in (5.4). Then $\sigma(a \cdot \alpha) = i^!(\sC)=c_{X_s}$, i.e $\sigma(\alpha) =1/a c_{X_s}$.  Let $\sZ \in A^2(\sX)$;  then $i^!(\sZ)=\sZ\vert X_s \in A_0(X_s)$ and  $\sZ\vert X_s = \sigma(Z_{\eta})$, where $Z_{\eta} = b \cdot \alpha$ is the class of $j^*(\sZ)$ in $A_0(X_{\eta})$. Therefore $\sigma(Z_{ \eta}) = b\cdot \sigma(\alpha) =(b/a)c_{X_s}$ 
that  proves 
$$\sZ\vert X_s \in \Q[c_{X_s}].$$
 We also  have $j^*(\sZ - (b/a)\sC) =0 $ in $A^2(X_{\eta})$, hence $(\sZ - (b/a)\sC) = i_*(\beta_s)$ with $\beta_s \in A_n(X_s)$. Here $ A_n(X_s) = 0$ for  $n >2$.\par
 \noindent  (ii)  $\Longrightarrow$ (i). Let $\alpha \in A_0( X_{\eta})_0$ and let $\sZ \in A^2( \sX) $ be such that $j^*(\sZ) = \alpha \in A_0(X_{\eta})$. Let $\sZ = m \sC + \beta$, with $m \in \Q$ and  $j^*(\beta)=0$.
  Then $i^!(\sZ) =mc_{X_s}$ in $A_0(X_s)$, for every   closed fibre  $X_s$.  $\sigma(\alpha) $ is a   0 -cycle  of degree 0 in $A_0(X_s)$, under the specialization map  $\sigma : A_0(X_{\eta} )\to A_0(X_s)$, because 
 $\deg(\alpha) =0$. Therefore we get $m=0$, $\sZ = \beta$ in $A^2(\sX)$ and $j^*(\sZ) =\alpha =0$.
\end{proof}
The following Corollary of Theorem 5.6 shows tha condition (i) in Theorem 5.6 holds true, for $g=3,4,5$, i.e.  In the cases  where the projective model in $\P^g$ of a general K3 surface of genus $g$ is a complete intersection of $g-2$ hypersurfaces.
\begin{cor} Let $f : \sX \to S$ be the family of polarized K3 surfaces of genus $g$, with $g =3,4,5$. Then  $A_0(X_{\eta}) \simeq \Q$, with $X_{\eta}$ the generic fibre of $f$.\end{cor}
\begin{proof}  Let  $g=3$. The universal family of K3 surfaces  of genus $g =3$ coincides with the family of quartic surfaces in $\P^3$.   Let $f: \sX \to  |\sO(4)|=S$ be the family of all quartic surfaces in $\P^3_{\C}$. Then $\sX$ is a projective  $\P^r$-bundle over $\P^3$, with  $r+1 =35$. As such $\sX$ has a finite dimensional motive, because its motive is isomorphic to the direct sum of   copies of the motive of $\P^3$. Also the Chow ring of $\sX$ is a finitely generated module over $A(\P^3)$ and hence it is a finite dimensional $\Q$- vector space. Therefore  $A^2(\sX_{U})$, where $U$ is a Zariski open subset of $S$, is a finite dimensional $\Q$-vector space and hence
$$ A_0(X_{\eta})=  \lim_{U \subset S} A^2(\sX_U) \simeq   \Q$$
Let $g=4$. Then the generic K3 of genus $g$ is the intersection of a smooth quadric $Q$ and a cubic in $\P^4$. Let $S = | \sO_Q(3)| $ and let $f : \sX \to S$. Then $\sX$ is a projective bundle over $Q$, with $\Q \simeq \P^1 \times \P^1$. Therefore, as in the previous case, the Chow ring of $\sX$ is a finitely generated module over $A(\P^1 \times \P^1)$ and hence it is a finite dimensional $\Q$- vector space. It follows that $A^2(\sX_U)$, for $U$ open in $S$ is a finite dimensional $\Q$ vector space and $A_0(X_{\eta}) \simeq \Q$.\par
\noindent $g =5$. The generic K3  is  the intersection of 3 quadrics in $\P^5$ and hence it is the base locus of  a 3-dimensional linear subspace of $| \sO_{\P^5}(2) |$. Let  $\sX $  be  the incidence variety  $Z \subset \Gr(3,H^0(\sO_{\P^5}(2)) \times \P^5$ given by couples $(P,x)$, where $x\in \P^5$ belongs to the intersection of the quadrics parametrized by $P\in \Gr(3,H^0(\sO_{\P^5}(2))$. 
Let  $S = \Gr(3,H^0(\sO_{\P^5}(2)$  and let   $f : Z \to S$ be the map induced by the first projection.  For every $ s \in S$ the fibre $f^{-1}(s) \subset Z$  is the K3 surface of all points in $\P^5$ lying on the 3 quadrics    parametrized by $s$. All Chow groups $A^k(\Gr(3,H^0(\sO_{\P^5}(2))$ of the Grassmanian bundle $\Gr(3,H^0(\sO_{\P^5}(2))$ are   isomorphic to a finite direct sum of Chow groups of $\P^5$ (see [Fu 14.6.5])  and hence they are  finite dimensional $\Q$-vector space. Also the Chow groups of the incidence sub variety $Z \subset  \Gr(3,H^0(\sO_{\P^5}(2)) \times \P^5$ are finite dimensional $\Q$-vector spaces. Therefore $A^2(\sX)$ and $A^2(\sX_{U})$, with $U$ a open subset of $S$, are finite dimensional $\Q$-vector spaces, so that, as in the previous cases, we get $A_0(X_{\eta} \simeq \Q$.
\end{proof}
\begin{rk} If $X$ is a general polarized K3 surface of genus $g$ with $g >5$ then the projective model of $X$ is not a complete intersection in  $\P^g$. However S.Mukai  proved, in [Mu 1]  and [Mu 2], that, for $6 \le g\le 10$, and also for $g =12,13, 18 ,20$,  $X$ is still a complete intersection with respect to a homogeneous vector bundle in a $g$-dimensional Grassmanian.  So an argument  similar to the one  used in the proof of the Corollary above may be used also in this cases. The following result takes care of the case $g=6$ \end{rk}
\begin {cor}  Let $f : \sX \to S$ be the universal family of polarized K3 surfaces of genus $g$, with $g =6$. Then  $A_0(X_{\eta}) \simeq \Q$, with $X_{\eta}$ the generic fibre of $f$.\end{cor}
\begin{proof} A  general polarized K3 surface of genus 6 can be obtained as a complete intersection in the Grassmanian $\Gr(2,5)$ of 2-dimensional subspaces in a fixed 5-dimensional vector space.  $G=\Gr(2,5)$ is embedded into $\P^9$ by Pl\"ucker coordinates and has dimension 6 and degree 5. The intersection in $\P^9$ of $G$ with  3  hyperplanes $H_1,H_2,H_3$ is a Fano 3-fold $F_5$ of index 2 and degree 5. The isomorphism class of $F_5$ does not depend on the choice of the 3 hyperplanes. A smooth complete intersection of $F_5$ with a quadratic hypersurface is a K3 surface of genus $g$. Let $\sV$ be the projective bundle $| 3\sO_G(1) \oplus \sO_G(2)|$ over $G$ and let $\sX$ be the incidence variety $Z \subset G \times I\sO_G(1) \oplus \sO_G(2)|$, given by couples $(x,P)$ where $x \in G$ belongs to the intersection of the hyperplanes and the quadratic hypersurfaces corresponding to $P \in I 3\sO_G(1) \oplus \sO_G(2)|$. Let $f : \sX \to G$ be the map induced by the first projection. The Chow ring $A(G)$ is a finite dimensional  $\Q$-vector space. 
Let $U \subset G$ be a Zariski open subset such that the projective bundle $\sV$ is trivial over $U$. Then $A^2(\sX_U) \simeq A^2(U \times (\sV)_U)$ is a finite dimensional $\Q$-vector space.
Therefore $A_0(X_{\eta}) \simeq \Q$, where  $X_{\eta} $ is the generic fibre of $f  : \sX \to G$.
  \end{proof}
The following lemma appears in [GG,6.1]. 
\begin{lm} Let $f : \sX \to C$ be a smooth projective family of surfaces over an algebraically  closed field of characteristic 0, with $\dim C=1$. Let $\eta$ be the generic point of $C$ and let $s $ be a closed point. Let $\bar K$ be the algebraic closure of $K =k(\eta)$ and let $X_{\bar K} =X_{\eta} \times_K  k(\eta)$, where $X_{\eta}$ is the generic fibre of $f$. Then if $A_0(X_{\bar K}) \simeq \Q$ also $A_0(X_s) \simeq \Q$, where $X_s$ is the fibre of $f$ over the closed point $s$.\end{lm}
\begin{proof} Let $R $ be the completion of the local ring of $C$ at $s$. Then, passing to colimits over finite extensions of $R$,  we have a specialization homomorphism over algebraically closed fields of characteristic 0 \par
\noindent  $ \sigma : A^i(X_{\bar K} )\to A^i(X_s)$ and a commutative diagram
$$ \CD  A^i(X_{\bar K} )@>{\sigma} >> A^i(X_s) \\
@V{cl}VV       @V{cl}VV    \\
H^{2i}(X_{\bar K}, \Q_l(i))@>{\simeq}>>H^{2i}(X_s, \Q_l(i)) \endCD$$
where $H^*$ is $l$-adic cohomology,  see [Fu, 20.3.5].  Now assume $A_0(X_{\bar K}) \simeq \Q$. Then, by the results in [BS] , $v(\Delta_{X_{\bar K}})=0$. Therefore  $t_2(X_{\bar K})=0$, because the identity map on $t_2(X_{\bar K})$ is 0 in 
$ \sM_{rat}(\bar K)$. It follows that  the motive $h(X_{\bar K})$ is finite dimensional. Also the group $H^2(X_{\bar K}, \Q_l(1))$ is algebraic. The finite dimensionality of the motive $h(X_{\bar K})$ in $\sM_{rat} (\bar K)$ implies, via the specialization map $\sigma$, the finite dimensionality of $h(X_s)$ in $\sM_{rat}(k)$, see [Ped 1, 4.3]. From the commutative diagram we also get the algebraicity
of $H^2(X_s), \Q_l(1))$.  Therefore $t_2(X_s)=0$ which implies $A_0(X_s) \simeq \Q$
\end{proof}
\begin {cor} Let $f : \sX \to \P^n_k$ be a smooth projective family of surfaces over an algebraically closed field $k$ of characteristic 0. Let $X_{\eta}$ be the generic fibre of $f$, which is a surface over $K = k(t_1, \cdots ,t_n)$, and let $X_P$ be the fibre over a closed point $P \in \P^n_k$. Assume that $A_0(X_{\bar K}) \simeq \Q$. Then also $A_0(X_P)\simeq \Q$.\end{cor}
\begin {proof} Let $\sM$ be the maximal ideal of $A = k[t_1,\cdots,t_n]$, corresponding to the closed point $P$, and let $A_{\sM} \simeq  k[t_1,\cdots,t_n]_{(t_1,\cdots,t_n)}$ be the local ring of $\P^n_k$ at $P$. Let $S =\Spec A_{\sM}$ and let $\tilde f:  \tilde \sX \to S$  be the induced fibration. Let $R_n = A_{\sM}/(t_1,\cdots,t_{n-1})A_{\sM}$ .  Then $R_n$ is a local ring of dimension 1 with quotient field  $k(t_n)$ and residue field $k$. Let $S_n =\Spec R_n$ and let
$$\CD  \tilde \sX_n@>>> \tilde\sX \\
@V{\tilde f_n}VV    @V{\tilde f}VV \\
S_n@>>> S \endCD $$
be the base change. The generic fibre  $(\tilde X_n)_{\eta_n}$ of $\tilde f_n$ is a surface over $k(t_n)$ and the closed fibre the surface $ X_p$ over $k$. Let $\bar K_n$ be the algebraic closure of the field $k(t_n)$, Then $\bar K_n \subset \bar K$ and $A_0(\tilde X_n)_{\bar K_n}) \simeq A_0 (\tilde X_{\bar K})\simeq \Q$, because the base change from an algebraically closed field to a larger algebraically closed field induces an isomorphism on Chow groups with $\Q$-coefficients. From lemma   we get $A_0(X_P)\simeq \Q$.\end{proof}
\begin {rk} Let $f : \sX \to S$ be one of the families of polarized K3 surfaces of genus $g$ considered in Corollary  5.7 and in Corollary 5.9,  where $S$ is either isomorphic or birational to $\P^n_{\C}$, for some $n$.  Let  $X_{\eta}$ be the generic fibre of $f$ and let $K$ be the algebraic closure of $k(\eta) =k(S)$. Then, from Corollary 5.11, we get $A_0(X_{K}) \ne \Q$, because  $A_0(X_s) \ne \Q$ for a closed fibre $X_s$, which is a K3 surface over $\C$.
 \end{rk}
 \begin {prop}  Let $f : \sX \to C$, with  $\dim C=1$,  be a smooth projective family of  K3  surfaces over an algebraically closed field  $k$. Then the following conditions are equivalent:\par
 (i)$A_0(X_{\eta} ) \simeq \Q$, where $\eta$ is the generic point of $C$;\par
 (ii)  $A^2(\sX)$ is a finitely dimensional $\Q$-vector space.\ r
  \end{prop}
 \begin {proof} (i)$\Longrightarrow$ (ii).  For every open subset $U =C - \{P_1,\cdots P_n\}$ the localization sequence in (5.4) gives
\begin{equation} \bigoplus_{1 \le i \le  n} A^1(X_{P_i}) \to A^2( \sX) \to A^2(\sX_U)\to 0\end{equation}
where the groups $A^1(X_{P_i}) =NS(X_{P_i})_{\Q}$ are finitely generated vector spaces. The category of Chow motives over the generic point $\eta$ of $C$ equals  the colimit of categories of Chow motives over non-empty  open subsets $U$ in the base curve $C$, i.e there is a functor F 
$$\CD  F : \sM_{rat} (\eta)@>{\simeq}>>  \colim_{U \subset C} \sM_{rat}(U) \\ \endCD $$
 which is an equivalence of categories, see [Gu, Lemma 4].  The functor $F$ is obtained by taking spreads of algebraic cycles and localizations sequences for Chow groups. If $M=(X,p)$ is  a Chow motive over $k(\eta)$ then $F(M) =(Y',p')$, where $Y'$ and $p'$ are spreads of $Y$ and $p$ over some open subset $U \subset C$. Also, by eventually shrinking $U$, we may assume that $p'$ is a projector, so that $M'=( Y',p')$ is a Chow motive over $k(U)$. On morphisms $F$ is defined in a similar way, because morphisms in a category of Chow motives are correspondences. We have
$$A^2(X_{\eta}) = \Hom_{\sM_{rat}(k(\eta)}( \un, h(X_{\eta})) = A_0 ( \Spec k(\eta) \times X_{\eta}) \simeq$$
$$\simeq \Q \simeq \Hom(_{\sM_{rat}(k(\eta)}(\Spec k(\eta) \times \Spec k(\eta)).$$
Therefore there exists an open subset $U \subset C$ such that $A^2(\sX_U) = \ Hom_{\sM_{rat}(k(U)}(\Spec k(U) \times \Spec k(U))$, i.e $A^2(\sX_U) \simeq \Q$.
From the localization sequence we get  that $A^2(\sX)$ is a finite dimensional  $\Q$-vector space.\par
(ii) $\Longrightarrow$ (i). From the exact sequence in (5.14) $A^2(\sX_U)$ is a finite dimensional $\Q$-vector space, for every open   $U \subset C$. From (5.5) we get
$\lim_{U \subset C} A^2(\sX_U)= A^2(X_{\eta})$ and hence $A_0(X_{\eta} \simeq \Q$.
 \end{proof} 
\begin{rk}Let $f : \sX \to C$  be a smooth projective family of surfaces on $k =\bar k$ with $C$ a smooth curve. In  [GG] it has been proved that, if  $A_0(\sX)\simeq \Q$, then $A^2(\sX)$ is a finite dimensional $\Q$- vector space and $\sX$ has a finite dimensional motive,  which lies in the subcategory of $\sM_{rat}(k)$ generated by the motives of abelian varieties. This is the case if $\sX$ is a Fano 3-fold.\end{rk}
\begin{ex} (Huybrechts) Let  $\sX \to \P^1$ be a general Lefschetz pencil of quartics in $\P^3_{k}$, with $k$ an algebraically closed field of characteristic 0, such that $\sX =Bl_C(\P^3)$ where $C$ is a complete intersection curve of  genus $g >0$  and $A^2(\sX) \simeq \Q \oplus \Pic C$.  Let $X_{\eta}$ be the generic fibre. Then $A_0(X_{\eta}) \ne \Q$ because  $A^2(\sX)$ is not a finite dimensional $\Q$ -vector space.\end{ex}


\begin{thebibliography}{10} 

 \bibitem [AC]{AC} E.Arbarello and M.Cornalba,{\it The Picard group of the moduli space of curves},Topology {\bf 26}(1987), 153-171
 
 \bibitem[BV]{BV} A. Beauville and C. Voisin 
{\it On the Chow ring of a K3 surface}, J.Alg geom. 13,(2004),417-426

\bibitem[BS]{BS} 
S. Bloch  and V.Srinivas
{\it Remarks on correspondences and algebraic cycles }, 
 American J.of Math {\bf 105} (1983) 1234-1253
 \bibitem[Fu]{Fu}
W. Fulton,
{\it Intersection Theory},
Springer-Verlag, Heidelberg-New-York, 1984.

\bibitem [GS]{GS}   A. Garbagnati and  A.Sarti,
{\it On symplectic and non-symplectic automorphisms of K3 surfaces}
Revista Matematica Iberoamericana 29 (2013), no. 1.
 
 \bibitem [GG]{GG} S.Gorchinskiy and V.Gulestkii {\it 
Motives and representability of algebraic cycles on threefolds over a field},
J.Alg. Geometry {\bf 21} (2012) 343-373.

\bibitem[Gu]{Gu}  V.Gulestkii {\it On the continuous part of codimension 2 algebraic cycles on three-dimensional varieties}, Sbornik Math. 200:3,(2009),325-338.
 
 
   \bibitem [Huy 1]{Huy 1} D.Huybrechts, 
{\it Symplectic automorphisms of K3 surfaces of arbitrary order}, 
Math.Res.Lett. {\bf 19} (2012) 947-951

 \bibitem [Huy 2]{Huy 2}  D.Huybrechts, (with an appendix by C. Voisin){\it  Curves and cycles on K3 surfaces} Algebraic Geometry 1 (2014), 69-106. arXiv:1303.4564

  \bibitem[Jan]{Jan}
U.Jannsen,{\it  Motivic sheaves and Filtrations on Chow groups}, Proceedings of Symposia in Pure Mathematics, vol.55,Part 1 (1994),245-302 
  
  \bibitem[KMP]{KMP} 
B. Kahn, J. Murre and C. Pedrini, 
\newblock {\it On the transcendental part of the motive of a surface}, 
pp. 143--202 in "Algebraic cycles and Motives Vol II", 
\newblock London Math. Soc. LNS {\bf 344}, Cambridge University Press, 2008. 

\bibitem [Ki]{Ki} S.I.Kimura {\it  Chow groups are finite dimensional, in some sense}, Math. Ann. {\bf 331} no. 1 (2005), 173-201  

  \bibitem [Muk 1 ]{Muk 1} S.Mukai {\it  Curves,K3 surfaces and Fano's 3-folds of genus $\le 10$}, Algebraic Geometry and Commutative algebra in Honor of M.Nagata, Academic Press (1987)
  355-377.
   \bibitem [Muk 2]{Muk 2} S.Mukai  {\it Polarized K3 surfaces of genus thirteen},Advances Studies in Pure Math. 45,(2006) 315-326
  
   \bibitem[MNP]{MNP} J.Murre, J.Nagel and C.Peters, 
{Lectures on the theory of pure Motives}, 
AMS University Lectures Ser. Vol 61 (2013)

\bibitem [Ni]{Ni} V. Nikulin 
{\it Finite automorphisms of Ka\"hlerian surfaces of type K3}, 
Trans. Moscow Math. Society, {\bf 38} (1980),71-137.

\bibitem[O'Gr]{O'Gr}  K.G.OÕGrady, {\it Moduli of sheaves and the Chow group of K3 surfaces}, J. Math. Pures Appl. 100 (2013), 701Ð718.

\bibitem [Ped 1]{Ped 1} C.Pedrini {\it On the motive of a K3 surface}, in "The geometry of algebrac cycles", Clay Math. Proceedings , Vol {\bf 9}, (2010), 53-73
  


\bibitem [Ped 2]{Ped 2} C.Pedrini {\it On the finite dimensionality of a K3 surface}, Manuscripta Math. {\bf 138}, (2012), 59-72
  

 \bibitem [PW 1]{PW 1} C.Pedrini and C.Weibel 
{\it Severi's results on correspondences}, Rend. Sem. Mat. Univ. Politec. Torino
Vol. 71, 3-4 (2013),493-504

\bibitem [PW 2]{PW 2} C.Pedrini and C.Weibel {\it Some surfaces of general type for which Bloch's conjecture holds} to appear on "Proceeding of the Conference on "Period Domains, Algebraic Cycles, and Arithmetic", Vancouver ,2013 Cambridge University Press

 \bibitem[Sev]{Sev} 
F. Severi, 
{\it La teoria delle corrispondenze a valenza sopra una superficie algebrica:
il principio di corrispondenza Nota III}, 
Rend. Reale Acc. Naz. Lincei, Classe di Scienze, Vol XVII (1933), 869-881.

  \bibitem [VG-S]{VG-S} B.Van Geemen and A. Sarti, {\it Nikulin involutions on K3 surfaces}, Math.Z. {\bf 255}, (2007),731-753
 
 \bibitem [Vois 1]{Vois 1}
C.Voisin,  
{\it Symplectic involutions on K3 surfaces act trivially on $CH_0$}, 
Documenta Math. 17 (2012) 851-860.

 \bibitem [Vois 2]{Vois 2}
C.Voisin,  
{\it  Bloch's conjecture for Catanese and Barlow surfaces} J. Differential Geometry 97 (2014) 149-175.
 
 
 \bibitem [Vois 3]{Vois 3}
C.Voisin,  {\it Rational equivalence of 0-cycles on K3 surfaces and conjectures of Huybrechts and O'Grady}preprint 2012, to appear in  ``Recent Advances in Algebraic Geometry,  a conference in honor of Rob Lazarsfeld's 60th birthday".  

 
 \bibitem [Vois 4]{Vois 4}
C.Voisin {\it Chow groups, decomposition of the diagonal and the topology of families},  Annals of Math. Studies 187,  Princeton University Press 2014.
  
\end{thebibliography}
\end{document}